\documentclass{amsart}
\usepackage{epsfig}
\usepackage{graphicx}
\usepackage{color}
\usepackage{epstopdf}

\def\R{\mathbb R}
\def\N{\mathbb N}

\def\const{\text{\upshape Const.}}

\def\negqquad{\!\!\!\!\!\!\!\!}
\def\eps{\varepsilon}

\newcommand{\cal}[1]{{\mathcal #1}}
\newcommand{\assref}[1]{{\rm (Q\ref{#1})}}
\newcommand{\fassref}[1]{{\rm (F\ref{#1})}}
\newcommand{\numera}{\renewcommand{\labelenumi}{(Q\arabic{enumi})}
\begin{enumerate}\setcounter{enumi}{\value{assumption}}}
\newcommand{\finenumera}[1]{\end{enumerate}\addtocounter{assumption}{#1}}
\newcounter{assumption}
\newcommand{\numeraf}{\renewcommand{\labelenumi}{(F\arabic{enumi})}
\begin{enumerate}\setcounter{enumi}{\value{assumptionf}}}
\newcommand{\finenumeraf}[1]{\end{enumerate}\addtocounter{assumptionf}{#1}}
\newcounter{assumptionf}

\newtheorem{thm}{Theorem}

\newtheorem{lem}{Lemma}

\theoremstyle{definition}
\newtheorem{definition}{Definition}
\theoremstyle{remark}
\newtheorem{rem}{Remark}

\begin{document}
\title[Blow--up for the wave equation with...]
{Blow--up for the wave equation with nonlinear source and boundary
damping terms}
\author{Alessio Fiscella}
\author{Enzo Vitillaro}
\address[E.~Vitillaro]
       {Dipartimento di Matematica ed Informatica, Universit\`a di Perugia\\
       Via Vanvitelli,1 06123 Perugia ITALY}
\email{enzo.vitillaro@unipg.it}
\address[A.~Fiscella]{Dipartimento di Matematica ``Federigo Enriques'', Universit\`a di Milano\\
Via Cesare Saldini, 50  20133 Milano ITALY}
\email{alessio.fiscella@unimi.it}
\date{\today}
\begin{abstract}
The paper deals with  blow--up  for the solutions of an evolution
problem consisting on a semilinear  wave equation posed in a bounded
$C^{1,1}$ open subset of $\R^n$, supplied with a Neumann boundary
condition involving a nonlinear dissipation.
 The typical problem studied is
$$\begin{cases} u_{tt}-\Delta u=|u|^{p-2}u \qquad &\text{in
$[0,\infty)\times\Omega$,}\\
 u=0\qquad &\text{on
$(0,\infty)\times\Gamma_0$,}\\ \partial_\nu u
=-\alpha(x)\left(|u_t|^{m-2}u_t+\beta |u_t|^{\mu-2}u_t\right) \qquad
&\text{on
$(0,\infty)\times\Gamma_1$,}\\
 u(0,x)=u_0(x),\qquad u_t(0,x)=u_1(x)&\text{in
$\Omega$,}
\end{cases}$$
where  $\partial\Omega=\Gamma_0\cup\Gamma_1$, $\Gamma_0\cap
\Gamma_1=\emptyset$, $\sigma(\Gamma_0)>0$, $2<p\le 2(n-1)/(n-2)$
(when $n\ge 3$), $m>1$, $\alpha\in L^\infty(\Gamma_1)$, $\alpha\ge
0$, $\beta\ge 0$. The initial data are posed in the energy space.
The aim of the paper is to improve  previous blow--up results
concerning the problem.
\end{abstract}

\maketitle

\section{Introduction}
We deal with the evolution problem consisting on a semilinear wave
equation posed in a bounded subset of  $\R^n$, supplied with a
Neumann boundary condition involving a nonlinear dissipation. More
precisely we consider the initial--and--boundary value problem
\begin{equation}\label{P}
\begin{cases} u_{tt}-\Delta u=f(x,u) \qquad &\text{in
$(0,\infty)\times\Omega$,}\\
 u=0\qquad &\text{on
$(0,\infty)\times\Gamma_0$,}\\
\partial_\nu u=-Q(x,u_t) \qquad &\text{on
$(0,\infty)\times\Gamma_1$,}\\
 u(0,x)=u_0(x),\qquad u_t(0,x)=u_1(x)&
 \text{in $\Omega$,}
\end{cases}\end{equation}
where $u=u(t,x)$, $t\ge 0$, $x\in\Omega$, $\Delta$ denotes the
Laplacian operator, with respect to the $x$ variable. We assume that
$\Omega$ is a bounded and $C^{1,1}$ open subset of $\R^n$ ($n\ge
1$), $\partial\Omega=\Gamma_{0}\cup\Gamma_{1}$,
$\Gamma_{0}\cap\Gamma_{1}=\emptyset$ with $\Gamma_0$ and $\Gamma_1$
being misurable with respect to the natural (Lebesgue) measure on
the manifold $\Gamma=\partial\Omega$, in the sequel denoted by
$\sigma$, and $\sigma(\Gamma_0)>0$. These properties of $\Omega$,
$\Gamma_0$ and $\Gamma_1$ are assumed, without further comments,
throughout the paper. The initial data are in the energy space, that
is $u_0\in H^1(\Omega)$ and $u_1\in L^2(\Omega)$, with the
compatibility condition ${u_0}_{|\Gamma_0}=0$ (in the trace sense).

Moreover $Q$ represents a nonlinear boundary damping and, roughly,
$Q(x,v)\simeq \alpha(x)(|v|^{m-2}v+\beta |v|^{\mu-2}v)$, $1<\mu\le
m$, $\beta\ge 0$, $\alpha\in L^\infty(\Gamma_1)$, $\alpha\ge 0$.
When $\beta>0$ and $\mu=2$  the term $Q$ describes a realistic
dissipation rate, linear for small $v$ and superlinear for large $v$
(see for example \cite{levicivita}), possibly depending on the space
variable, while when $\beta=0$ and $\alpha=1$ it is a pure--power
model nonlinearity. Finally $f$ is a nonlinear source and roughly
$f(x,u)\simeq |u|^{p-2}u$, $2<p\le 2^*$, where as usual $2^*$ denotes
the Sobolev critical exponent $2n/(n-2)$ when $n\ge 3$, $2^*=\infty$
when $n=1,2$.

The presence of the boundary damping in \eqref{P} plays a critical
role in the context of boundary control. See for example
\cite{chen2}, \cite{chen4}, \cite{chen3}, \cite{chen1},
\cite{lagnese2}, \cite{lasiecka1}, \cite{lastat},
\cite{lasieckatriggiani} and \cite{zuazua}. For this reason, and for
their clear physical meaning, problems like \eqref{P} are subject of
a wide literature. In addition to the already quoted papers see also
\cite{CDCL}, \cite{CDCM}, \cite{cavalcantisoriano},
\cite{chueshovellerlasiecka}, \cite{CLT},  \cite{gerbi}, \cite{Ha},
\cite{lasieckatriggiani2}, \cite{lasieckatriggiani3}, \cite{phy},
\cite{tataru} and \cite{stable}.

The analysis of problems like \eqref{P} is related to the treatment
of quasilinear wave equations  with Neumann boundary conditions involving source terms.
See \cite{bociulasiecka2}, \cite{bociulasiecka1}, \cite{BRT},
\cite{bociulasiecakproc}, \cite{PH} and \cite{global}.

In order to clearly describe the specific subject of this paper we
consider problem \eqref{P} when $f$ and $Q$ are exactly the model
nonlinearities, that is when problem \eqref{P} reduces to
\begin{equation}\label{2}
\begin{cases} u_{tt}-\Delta u=|u|^{p-2}u \qquad &\text{in
$(0,\infty)\times\Omega$,}\\
 u=0\qquad &\text{on
$(0,\infty)\times\Gamma_0$,}\\
\partial_\nu u=-\alpha(x) (|u_t|^{m-2}u_t+\beta |u_t|^{\mu-2}u_t) \qquad &\text{on
$(0,\infty)\times\Gamma_1$,}\\
 u(0,x)=u_0(x),\qquad u_t(0,x)=u_1(x)&
 \text{in $\Omega$,}
\end{cases}\end{equation}
with  $1<\mu\le m$, $\beta\ge 0$, $\alpha\in L^\infty(\Gamma_1)$,
$\alpha\ge 0$ and $2<p\le 2^*$.

Local existence and uniqueness for weak solutions of problem
\eqref{2} when $2<p\le 1+2^*/2$ was first proved in
\cite[Theorem~4]{stable}, see Theorem~\ref{localexistencetheorem},
p.~\pageref{localexistencetheorem}. In the literature one often
refer to this parameter range as the subcritical/critical one, since
the Nemitskii operator $u\mapsto |u|^{p-2}u$ is locally Lipschitz
from $H^1(\Omega)$ to $L^2(\Omega)$. In this case the nonlinear
semigroup theory is directly available.

The quoted result was
subsequently extended to more general nonlinearities $Q$ and $f$, of
non--algebraic type, in \cite{CDCL} and \cite{CDCM}. Moreover, at
least when $\alpha$ is constant, H\"{a}damard well--posedness for
problem \eqref{2} follows from the results in \cite{bociulasiecka1},
dealing with more general versions of problem \eqref{P} possibly
involving internal nonlinear damping and boundary source terms. On
this concern it is worth observing that, when no internal damping is
present in the equation, the well--posedness result in
\cite{bociulasiecka1} only applies to the subcritical/critical range
$2<p\le 1+2^*/2$, due to \cite[Assumption 1.1]{bociulasiecka1}.
Moreover, when $u_0$ and $u_1$ are small (in the energy space) the
solutions of \eqref{2} are global in time.

On the other hand blow--up results for problem \eqref{2} are much
less frequent in the literature. In the particular case
$\Gamma_1=\emptyset$ (the same arguments work also when
$\alpha\equiv 0$) it is well--known that, for particularly chosen
data,  local solutions of problem \eqref{2}, when they exist,
blow--up in finite time. See for example \cite{ball3},
\cite{glassey}, \cite{jorgens}, \cite{keller}, \cite{klp},
\cite{levine4}, \cite{levine2} and  \cite{tsutsumi}. We also refer
to the related papers \cite{levpayne1} and \cite{levpayne2}, dealing
with boundary source terms. In \cite{paynesattinger} the authors
introduced the so called ``potential well theory" for semilinear
wave equation with Dirichlet boundary condition, and in particular
blow--up for positive initial energy was proved. We also would like
to mention the paper \cite{georgiev}, dealing with the equation
 $u_{tt}-\Delta u+|u_t|^{m-2}u_t=|u|^{p-2}u$ in
$[0,\infty)\times\Omega$ with homogeneous Dirichlet boundary
conditions,  when $2<p\le 1+2^*/2$ and $m>1$, which was the first
contribution facing  the competition between nonlinear damping and
source terms. In particular it was there proved that solutions may
blow--up in finite time (depending on initial data) if and only if
$m<p$. The result was subsequently generalized to positive initial
energy and abstract evolution equations in several papers. See for
example \cite{levserr}, \cite{ps:private}  and \cite{blowup}.

When $\Gamma_1\not=\emptyset$ and $m=2$ the problem of global
nonexistence for solutions of \eqref{2} was studied in
\cite{rendiconti} using the classical concavity method of H. Levine,
which is no longer available for nonlinear damping terms. The first
blow--up result for problem \eqref{2}  in the general case $m>1$
(and $2<p\le 1+2^*/2$) is contained in the already quoted paper
\cite{stable}. To recall it we need to introduce some basic
notation.  We denote by  $\|\cdot\|_p$  the norm in $L^p(\Omega)$ as
well as the norm in  $[L^p(\Omega)]^n$.  We also introduce the
Hilbert space
$$H^1_{\Gamma_0}(\Omega)=\{u\in H^1(\Omega): u_{|\Gamma_0}=0\}$$
(where $u_{|\Gamma_0}$ is intended in the trace sense), equipped
with the norm $\|\nabla u\|_2$, which is equivalent, by a Poincar\`e
type inequality (see \cite{ziemer}), to the standard one. We also
introduce the functionals
\begin{equation}\label{J}
J(u)=\frac 12 \|\nabla u\|_2^2-\frac
1p\|u\|_p^p\qquad\text{and}\quad K(u)=\|\nabla u\|_2^2-\|u\|_p^p
\end{equation}
for $u\in H^1_{\Gamma_0}(\Omega)$. The energy associated to initial
data $u_0\in H^1_{\Gamma_0}(\Omega)$ and $u_1\in L^2(\Omega)$ is
denoted by $E(u_0,u_1):=\frac 12 \|u_1\|_2^2+J(u_0)$. Moreover we
set
\begin{equation}\label{d}
d =\underset{u\in
H_{\Gamma_0}^{1}(\Omega)\setminus\{0\}}{\inf}\sup\limits_{\lambda
>0}J(\lambda u).
\end{equation}
It is well--known that $d>0$. See Section~\ref{section4}, where
Lemma~\ref{NUOVO} makes clear this property, and also
Remark~\ref{variational}, where a variational characterizations of
$d$ is recalled. Finally  we introduce the
"bad part of the potential well" (we owe this suggesting name to \cite{BRTCOMP})
\begin{equation}\label{Wu}
W_u:=\{(u_0,u_1)\in H^1_{\Gamma_0}(\Omega)\times L^2(\Omega):
K(u_0)\le 0 \quad\text{and} \quad E(u_0,u_1)<d\}.
\end{equation}
Trivially if $E(u_0,u_1)<0$ then  $(u_0,u_1)\in W_u$ since $p>2$.
The situation is clearly described by  Figure~\ref{fig1}  below.

In particular \cite[Theorem~7]{stable} asserts that solutions
blow--up in finite time  if $(u_0,u_1)\in W_u$  and the further
condition
\begin{equation}\label{3} m<m_0(p):=\frac {2(n+1)p-4(n-1)}{n(p-2)+4}
\end{equation}
holds. It is worth mentioning that $m_0(p)>2$ when $p>2$, so the
case $1<m\le 2$ is fully covered, but when $m>2$ condition \eqref{3}
is rather restrictive. See Figure~\ref{fig2} below.

\begin{figure}
\includegraphics[width=17cm]{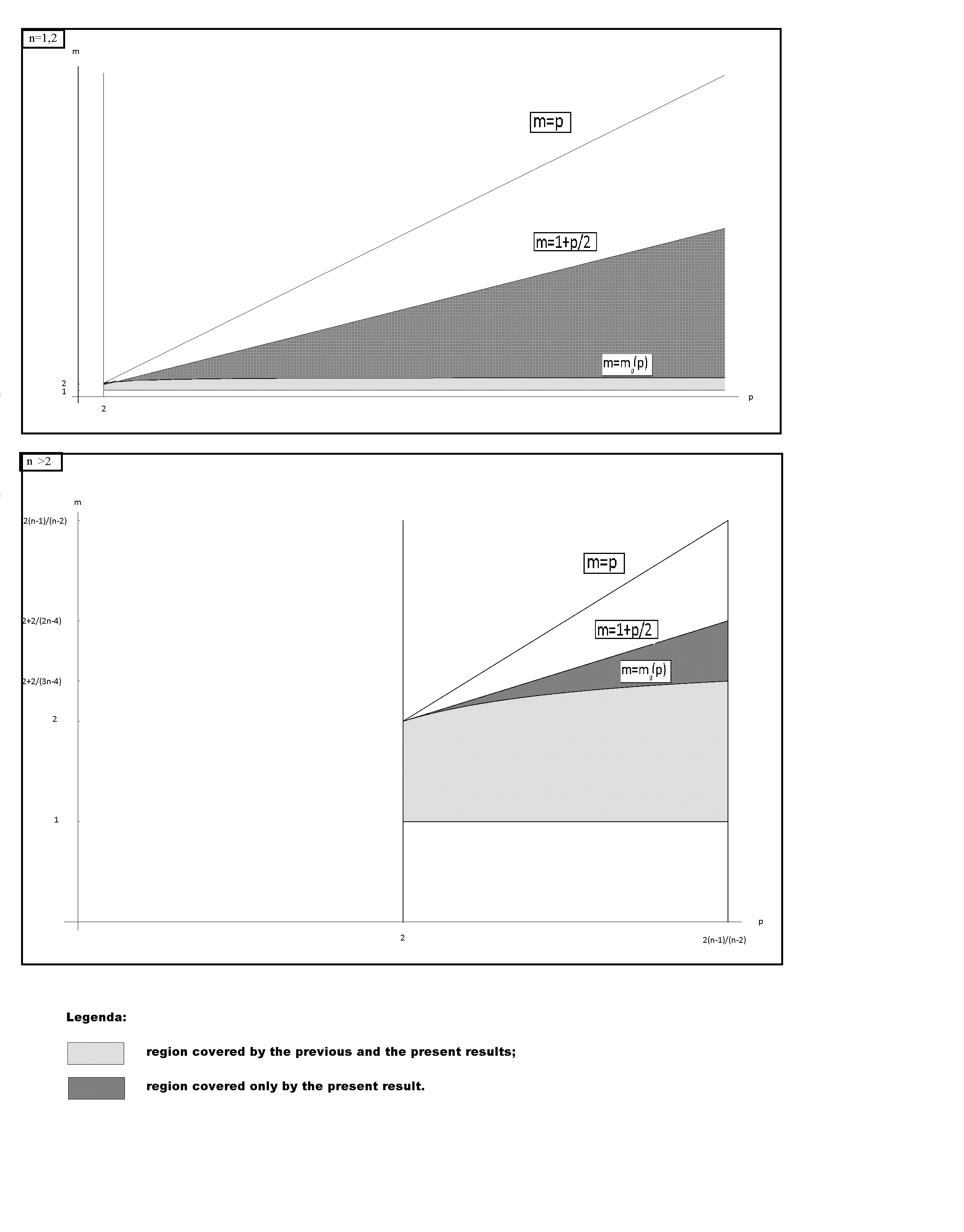}
\caption{The sets of
the $(p,m)$ considered in \cite{stable} and in the present paper, in
the two cases $n=1,2$ and when $n\ge 3$. The figure are made when
$n=2$ and $n=3$ in different scales due to the unboundedness of the
sets considered in the first case.} \label{fig2}
\end{figure}

In  \cite{BRT} and  \cite{CDCL} (also) the blow--up problem is considered.
These papers deal with a modified version of \eqref{2}, where also
internal damping and boundary source terms are present. Assumption
\eqref{3} is absent there, since the combination of internal and
boundary source is more effective in producing blow--up.

As to problem \eqref{2} without boundary sources we mention the
paper \cite{gerbi} where exponential growth, but non blow--up, for
solutions of \eqref{2} is proved when $m<p$. A generalized version
of assumption \eqref{3} also appears in the recent paper
\cite{autuoripucci}, dealing with much more general Kirchhoff
systems and a larger class of initial data.

Assumption \eqref{3} was first skipped in
\cite{bociulasieckaApplmat}, where blow--up for a modified version
of problem \eqref{2} is proved when $m<1+p/2$ and $E(u_0,u_1)<0$.
Even if the blow--up result in the quoted paper is stated in
presence of an internal damping, one easily sees that the arguments
in the proof apply as well to problem \eqref{2}. Clearly assumption
$m<1+p/2$ is more general than \eqref{3}, since $m_0(p)<1+p/2$ for
$p>2$  (see Figure~\ref{fig2} again). The improvement in the
assumption was obtained by using interpolation estimate in the full
scale of Besov spaces instead than in the Hilbert scale used in
\cite{stable}.

\begin{figure}
\includegraphics[width=12cm]{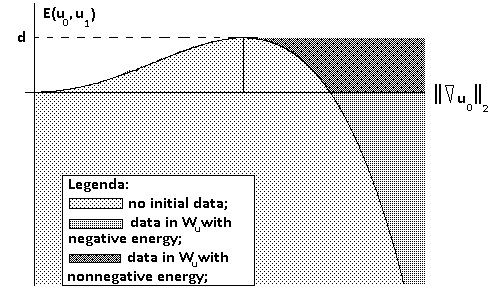}
\caption{The sets of initial data considered by
\cite{bociulasieckaApplmat}, having negative initial energy, and
those considered only in the present paper.} \label{fig1}
\end{figure}

Subsequently  assumption \eqref{3} was skipped also in the recent
paper \cite{FLZ}, dealing with the one--dimensional case $n=1$, when
$\beta=0$ and $\alpha\equiv 1$. Blow--up for problem \eqref{2} is
proved there when $E(u_0,u_1)<0$ and
\renewcommand{\labelenumi}{(\roman{enumi})}
\begin{enumerate} \item  either $m<1+p/2$,
\item or $m\ge 1+p/2$ and  $|\Omega|$ is sufficiently large.
\end{enumerate}
The arguments used by the authors in the two cases are different.
Consequently in dimension one the line $m=p$ is not the threshold
between global existence and blow--up for suitable data. A natural
conjecture is then that the same phenomenon occurs in higher space
dimension $n$, even if the one--dimensional case is sometimes
different from the higher--dimensional one (see for example the
papers \cite{vazvit2} and \cite{vazvit} where a similar situation
occurs for well--posedness, and the related paper \cite{vazvitHLB}).
Unfortunately the arguments used to handle with the case $m\ge
1+p/2$ cannot be adapted to $n\ge 2$.

The aim of this paper is to show that the technique in \cite{stable}
can be adapted to cover at least the case $m<1+p/2$. In this way we
extend the blow--up result from \cite{bociulasieckaApplmat} to
positive initial energy.  Instead of using interpolation theory we
adapt a more elementary estimate, used in \cite{FLZ} when $n=1$, to
the case $n\ge 1$.

Our main result concerning problem \eqref{2} is the following one.

\begin{thm}\label{theorem1} Let $\alpha\in L^\infty(\Gamma_1)$,
$\alpha\ge 0$, $\beta\ge 0$,
$$2<p\le 1+2^*/2,\qquad 1<m<1+p/2$$ and
$(u_0,u_1)\in W_u$. Then the weak solution $u$ of problem \eqref{2}
blows--up in finite time, that is  there is $T_{max}<\infty$ such
that $\|u(t)\|_p\to \infty$ (and so also $\|u(t)\|_\infty\to \infty$
and $\|\nabla u(t)\|_2\to \infty$) as $t\to T_{max}^-$.
\end{thm}

\begin{rem} The meaning of weak solutions will be made precise in the sequel. Moreover, it will be clear
(after the proof) that
the parameter range $2<p\le 1+2^*/2$ in Theorem~\ref{theorem1} can
be extended to $2<p\le 2^*$, but when $1+2^*/2<p\le 2^*$ we merely
obtain global nonexistence of weak solutions, since a local
existence theorem is missing.
\end{rem}

 The paper is organized as follows. In Section~\ref{section2}  we
recall (from \cite{stable}) our main assumptions,  local existence
and potential-- well theories for problem \eqref{P}, with some
additional remarks.  Section~\ref{section3} is devoted to state and
prove our main result, that is Theorem~\ref{theorem 4}, on  problem
\eqref{P}. In Section~\ref{section4} we show that, when applying
Theorem~\ref{theorem 4} to problem~\eqref{2}, we obtain
Theorem~\ref{theorem1}.
\bigskip

\section{Preliminaries}\label{section2}
\bigskip
\noindent In this section we recall some material from
\cite{stable}, referring to the quoted paper for most of the proofs.
We start by recalling the assumptions on $Q$ and $f$ needed for
local existence. \numera
\item \label{Q1} $Q$ is a Carath\'eodory real function in
$\Gamma_1\times\R$, and there are  $\alpha\in L^1(\Gamma_1)$,
$\alpha \ge 0$
\begin{footnote}{the integrability of $\alpha$ on $\Gamma_1$, although not explicitly assumed in \cite[Theorem~4]{stable}, was tacitely used there.}\end{footnote},
 and an exponent $m>1$ such that, if $m\ge 2$,
$$\left(Q(x,v)-Q(x,w)\right)(v-w)\ge \alpha(x)|v-w|^m$$ for all
$x\in\Gamma_1$, $v,w\in\R$, while, if $1<m<2$,
$$\left(Q(x,v)-Q(x,w)\right)(v-w)\ge \alpha(x)
\left||v|^{m-2}v-|w|^{m-2}w\right|^{m'}$$ for all $x\in\Gamma_1$,
$v,w\in\R$, where $1/m+1/m'=1$;

\item \label{Q2} there are $1<\mu\le m$ and $c_1>0$ such that
$$|Q(x,v)|\le c_1\alpha(x)\left(|v|^{\mu-1}+|v|^{m-1}\right)$$ for all
$x\in\Gamma_1$,  $v\in\R$.

\finenumera 2

\begin{rem}\label{remnew1} The model nonlinearity
\begin{equation} \label{Q0}
Q_0(x,v)=\alpha(x)\left(|v|^{\mu-2}v+|v|^{m-2}v\right),\qquad
1<\mu\le m,\quad \alpha\ge 0,\quad \alpha\in L^1(\Gamma_1),
\end{equation}
satisfies \assref{Q1} and \assref{Q2}. Indeed, while \assref{Q2} is
trivially verified, assumption \assref{Q1} holds, when $m\ge2$, up
to multiplying $\alpha$ by an inessential positive constant, due to
the elementary inequality
\begin{equation}\label{elementary2}
(|v|^{m-2}v-|w|^{m-2}w)(v-w)\ge \const |v-w|^m,\qquad v,w\in\R.
\begin{footnote}{which is a consequence of the boundedness of the
real function $(|t-1|^{m-2}(t-1))/(|t|^{m-2}t-1)$ when $m\ge 2$.
}\end{footnote}
\end{equation}
When $1<m<2$ we get \assref{Q1} by applying \eqref{elementary2} to
$m'>2$, $|v|^{m-2}v$ and $|w|^{m-2}w$.
\end{rem}

We  note, for a future use, some consequences of
\assref{Q1}--\assref{Q2}. First of all it follows that
\begin{equation}\label{low}
Q(x,v)v\ge \alpha (x)|v|^m
\end{equation}
for all $x\in\Gamma_1$, $v\in\R$. Moreover $Q(x,\cdot)$ is
increasing for all $x\in\Gamma_1$, and  $Q(\cdot,0)\equiv 0$. Then,
after  setting
\begin{equation}\label{PHI}
\Phi(x,u)=\int_0^uQ(x,s)\,ds,
\end{equation}
we obtain
\begin{equation}\label{philow}
\Phi(x,u)\ge \frac {\alpha(x)}m |v|^m\qquad\text{for all $x\in
\Gamma_1$, $v\in\R$}. \end{equation}

We now introduce some notation. When $1<q\le \infty$ we denote by
$L^q(\Gamma,\alpha)$ the $L^q$ space on $\Gamma$ associated to the
measure $\mu_\alpha$ defined by
$\mu_\alpha(A)=\int_A\alpha(x)\,d\sigma $ for any measurable subset
$A$ of $\Gamma$, while  $L^q(\Gamma)$ denotes the standard $L^q$
space, that is $L^q(\Gamma)=L^q(\Gamma,1)$. The analogous convention
will be adopted on $\Gamma_1$ and in $(0,T)\times\Gamma_1$ for $T>0$
(in the latter case the measure $\mu_\alpha$ being replaced by
$dt\times\mu_\alpha$). Moreover we shall write for simplicity
\begin{alignat*}2
&\|\cdot\|_{q,\Gamma,\alpha}:=\|\cdot\|_{L^q(\Gamma,\alpha)},\qquad
&&\|\cdot\|_{q,\Gamma}:=\|\cdot\|_{L^q(\Gamma)},\\
&\|\cdot\|_{q,\Gamma_1,\alpha}:=\|\cdot\|_{L^q(\Gamma_1,\alpha)},\qquad
&&\|\cdot\|_{q,\Gamma_1}:=\|\cdot\|_{L^q(\Gamma_1)}.
\end{alignat*}
Our assumption concerning $f$ is the following one: \numeraf
\item \label{F1} $f$ is a Carath\'eodory real function in
$\Omega\times\R$, $f(x,0)=0$ and there are $p>2$ and $c_2>0$ such
that
$$|f(x,u)-f(x,v)|\le c_2 |u-v|(1+|u|^{p-2}+|v|^{p-2})$$
for all $x\in\Omega$, $u,v\in\R$.

\finenumeraf 1

\begin{rem}\label{remnew2}
The model nonlinearity
\begin{equation}\label{f0}
f_0(x,u)=a|u|^{q-2}u+b|u|^{p-2}u,\qquad 2\le q<p,\qquad a,b\in\R,
\end{equation}
satisfies \fassref{F1}, due to the elementary inequality
$$\big||u|^{s-2}u-|v|^{s-2}v\big|\le \const |v-w|(1+|u|^{s-2}+|v|^{s-2}),\qquad u,v\in\R,$$
which holds for $s\ge 2$.
\end{rem}
We make precise the definition of weak solution used (somewhat
implicitly) in \cite{stable}.
\begin{definition}\label{def1}
When \assref{Q1},\assref{Q2}, \fassref{F1} hold and $2<p\le 2^*$ we
say that $u$ is a weak solution of problem \eqref{P} in $[0,T]$,
$T>0$, if
\renewcommand{\labelenumi}{(\alph{enumi})}
\begin{enumerate}
\item $u\in C([0,T];H^{1}_{\Gamma_{0}}(\Omega))\cap C^1([0,T];L^{2}(\Omega));$

\item the spatial trace of $u$ on
$(0,T)\times\Gamma$ (which exists by the trace theorem) has
a distributional time derivative on $(0,T)\times\Gamma_{1}$,
belonging to \mbox{$L^{m}((0,T)\times\Gamma_{1},\alpha)$};

\item for all $\varphi\in C([0,T];H^{1}_{\Gamma_{0}}(\Omega))\cap C^{1}([0,T];L^{2}(\Omega))\cap L^{m}((0,T)\times\Gamma_1,\alpha )$ and for almost all $t\in [0,T]$ the
distribution identity
\begin{equation}\label{ovo}
\int_\Omega u_t\varphi\Big|_0^t=\int_0^t\int_\Omega
u_t\varphi_t-\nabla u\nabla\varphi +\int_0^t\int_{\Omega}
f(\cdot,u)\varphi-\int_0^t\int_{\Gamma_1} Q(\cdot,u_t)\varphi
\end{equation}
 holds true;
\item $u(0)=u_{0}$ and $u_{t}(0)=u_{1}$.
\end{enumerate}
We say that $u$ is a weak solution of problem \eqref{P} in $[0,T)$
if $u$ is a weak solution in $[0,T']$ for all $T'\in (0,T)$. Finally
we say that a weak solution $u$ in $[0,T)$ is  maximal if $u$ cannot
be seen as a restriction of a weak solution in $[0,T')$, $T<T'$.
\end{definition}

\begin{rem} The term $\int_0^t\int_{\Omega}
f(\cdot,u)\varphi$ in \eqref{ovo} makes sense by \fassref{F1}, the
continuity of Nemitski operators and Sobolev embedding theorem. To
recognize that  the last term in the right--hand side of \eqref{ovo}
makes sense requires some attention. At first we note that, by (b),
we have $\alpha^{1/m}u_t\in L^m((0,T)\times\Gamma_1)$ and then
$\alpha^{1/m'}|u_t|^{m-1}\in L^{m'}((0,T)\times\Gamma_1)$. Since
$\varphi \in L^{m}((0,T)\times\Gamma_1,\alpha )$ we have
$\alpha^{1/m}\varphi \in L^{m}((0,T)\times\Gamma_1)$. Consequently
$\alpha |u_t|^{m-1}\varphi\in L^1((0,T)\times\Gamma_1)$. Now, since
$\mu_\alpha(\Gamma_1)<\infty$ and $\mu\le m$, we have
$L^m((0,T)\times\Gamma_1,\alpha)\subset
L^{\mu}((0,T)\times\Gamma_1,\alpha)$, hence we can repeat previous
arguments with $\mu$ instead of $m$ to show that  $\alpha
|u_t|^{\mu-1}\varphi\in L^1((0,T)\times\Gamma_1)$. Consequently, by
\assref{Q2} we get $Q(\cdot,u_t)\varphi\in
L^1((0,T)\times\Gamma_1)$.
\end{rem}

\begin{rem} \label{shift} We remark, for the sake of clearness, the following facts.
Since the equation and boundary conditions in problem \eqref{P} are
autonomous, the choice of the initial time as zero is purely
conventional. Consequently, for any $a\in\R$, we shall speak of weak
solutions in $[a,a+T]$, $T>0$, of the problem
\begin{equation}\label{Pa}
\begin{cases} u_{tt}-\Delta u=f(x,u) \qquad &\text{in
$(a,\infty)\times\Omega$,}\\
 u=0\qquad &\text{on
$(a,\infty)\times\Gamma_0$,}\\
\partial_\nu u=-Q(x,u_t) \qquad &\text{on
$(a,\infty)\times\Gamma_1$,}\\
 u(a,x)=u_0(x),\qquad u_t(a,x)=u_1(x)&
 \text{in $\Omega$,}
\end{cases}\end{equation}
when (a--d) in Definition~\ref{def1} hold true with $0$ and $T$
respectively replaced by $a$ and $a+T$. Moreover
\renewcommand{\labelenumi}{\roman{enumi})}
\begin{enumerate}
  \item the function $u$ is a weak
solution of \eqref{P} in $[0,T]$ if and only if the time shifted
function $\tau_a u$ defined by
\begin{equation}\label{ua}
    (\tau_a u)(t):=u(t-a)
\end{equation}
 is a weak solution of \eqref{Pa} in $[a,a+T]$;
  \item let $b\in \R$, $0<T_1<T_2$, $u_1$ be a weak
solution in $[b,b+T_1]$ of problem \eqref{Pa} with $a=b$ and $u_2$
be a weak solution in $[b+T_1,b+T_2]$ of problem \eqref{Pa} with
$a=b+T_1$. Define $u$ in $[b,b+T_2]$ by $u(t)=u_1(t)$ for $t\in
[b,b+T_1]$ and $u(t)=u_2(t)$ for $t\in (b+T_1,b+T_2]$. Then $u$ is a
weak solution of \eqref{Pa} with $a=b$ in $[b,b+T_2]$ if and only if
$u_1(b+T_1)=u_2(b+T_1)$ and $(u_1)_t(b+T_1)=(u_2)_t(b+T_1)$.
\end{enumerate}
\end{rem}

We now recall \cite[Theorem~4]{stable}.

\begin{thm}\label{localexistencetheorem} Suppose that \assref{Q1}--\assref{Q2} and
\fassref{F1} hold, that $2<p\le 1+2^*/2$, and $u_0\in
H^1_{\Gamma_0}(\Omega)$, $u_1\in L^2(\Omega)$. Then there is $T>0$
and a unique weak solution of \eqref{P} in $[0,T]$. Moreover $u$
satisfies  the energy identity
\begin{equation}\label{EI}
E(t)-E(s)=-\int_s^t\int_{\Gamma_1}Q(\cdot,u_t)u_t
\end{equation}
for $0\le s\le t$, where
\begin{align}
E(t)=E(u(t),u_t(t))=&\frac 12 \|u_t(t)\|_2^2+\frac 12\|\nabla
u(t)\|_2^2- \int_\Omega F(\cdot,u(t)),\label{eff}\\ \intertext{and}
F(x,s)=&\int_0^s f(x,\tau)\,d\tau\qquad\text{for $x\in\Omega$,
$s\in\R$.}\label{F}
\end{align}
\end{thm}

\begin{rem} Actually Theorem~\ref{localexistencetheorem} was stated in \cite{stable} for
regular (i.e. $C^1$)  domains, but one immediately sees that
$\Omega$ can be also disconnected (even if this case is not of
particular interest).
\end{rem}

As a consequence of the arguments used in the proof of
Theorem~\ref{localexistencetheorem} it follows the following
continuation principle, which was used in the quoted paper without
an explicit proof. For the sake of clearness we prefer to give here
its proof.

\begin{thm}\label{continuation}
Suppose that \assref{Q1}--\assref{Q2} and \fassref{F1} hold, that
$2<p\le 1+2^*/2$, and $u_0\in H^1_{\Gamma_0}(\Omega)$, $u_1\in
L^2(\Omega)$. Then \eqref{P} has a unique weak maximal solution $u$
in $[0,T_{max})$. Moreover the following alternative holds:
\renewcommand{\labelenumi}{(\roman{enumi})}
\begin{enumerate}
    \item  either $T_{max}=\infty$;
    \item  or $T_{max} < \infty$ and $\lim\limits_{t \rightarrow T^{-}_{max}}
\|u(t)\|_{H^{1}_{\Gamma_{0}(\Omega)}}+\|u_t(t)\|_2 = \infty$.
\end{enumerate}
\end{thm}
\begin{proof}
By the arguments in the proof of Theorem~\ref{localexistencetheorem}
it easily follows that the assured existence time $T$ depends on the
initial data $u_0$ and $u_1$ as a decreasing function of
$\|u_0\|_{H^1_{\Gamma_0}(\Omega)}^2+\|u_1\|_2^2$, which is in the
sequel denoted by
$$T^*=T^*(\|u_0\|_{H^1_{\Gamma_0}(\Omega)}^2+\|u_1\|_2^2).$$
From this remark the statement follows in a standard way. More
precisely we first construct the unique maximal solution $u$ as
follows. We set $\cal{U}$ to be the set of all weak solutions of
\eqref{P} in right--open intervals $[0,T')$, $T'>0$.

Then we claim that for any couple $u$, $v$ of elements of $\cal{U}$,
weak solutions respectively in $[0,T_u)$ and $[0,T_v)$, $u=v$ in the
intersection $[0,T)$ of their domains. To prove our claim we set
\begin{equation}\label{t0}
t_0:=\sup\{t\in [0,T): u(s)=v(s)\quad\text{for all}\quad s\in
[0,t)\},
\end{equation}
so $t_0\le T$. Now we suppose by contradiction that $t_0<T$. Since
$$u,v\in C([0,t_0];H^{1}_{\Gamma_{0}}(\Omega))\cap
C^1([0,t_0];L^{2}(\Omega))$$ we easily get that $u(t_0)=v(t_0):=v_0$
and $u_t(t_0)=v_t(t_0):=v_1$. Now since $u,v$ are weak solutions
(see Remark~\ref{shift}) of \eqref{Pa} with $a=t_0$ and initial data
$v_0$, $v_1$,  we see  that $\tau_{-t_0}u$ and $\tau_{-t_0}v$
(defined in \eqref{ua}) are both weak solutions in $[0,T-t_0)$ of
\eqref{P} with initial data $v_0$ and $v_1$. Hence, by the
uniqueness assertion in Theorem~\ref{localexistencetheorem} we get
that $\tau_{-t_0}u=\tau_{-t_0}v$ in $[0,T'']$,
$T''=T^*(\|v_0\|_{H^1_{\Gamma_0}(\Omega)}^2+\|v_1\|_2^2)>0$.
Consequently $u=v$ in $[0,t_0+T'']$, contradicting \eqref{t0}. Hence
$t_0=T$ proving our claim. To construct the maximal weak solution we
define $u$ to coincide with any element of $\cal{U}$ in the union of
the domains.

We now have to prove the alternative in the statement.  We suppose,
by contradiction, that
\begin{equation}\label{alternative} T_{\max} < \infty \quad
\mbox{and  } \liminf_{t\rightarrow
T^{-}_{\max}}\left(\|u(t)\|_{H^{1}_{\Gamma_{0}}(\Omega)}+\|u_t(t)\|_2\right)
< \infty.
\end{equation}
 Then there is a sequence $t_{n}\rightarrow
T^{-}_{\max}$ such that
$\left\|u(t_n)\right\|_{H^{1}_{\Gamma_{0}}(\Omega)}$ and
$\|u_t(t_n)\|_2$ are bounded, so $M:=\sup\limits_n
\left(\|u(t_n)\|_{H^1_{\Gamma_0}(\Omega)}^2+\|u_t(t_n)\|_2^2\right)<\infty$.
By Theorem~\ref{localexistencetheorem} and the monotonicity of $T^*$
asserted before for each $n\in\N$ the problem \eqref{P} with initial
data $u(t_n)$ and $u_t(t_n)$ has a unique weak solution $v_n$ in
$[0,T_1]$, $T_1=T^*(M)$. Hence, for each $n\in\N$,
$w_n=\tau_{t_n}v_n$ is a weak solution of \eqref{Pa} in
$[t_n,t_n+T_1]$ with $a=t_n$ and initial data $u(t_n)$ and
$u_t(t_n)$. It follows (see Remark \ref{shift}) that $u$ can be
extended to a weak solution of \eqref{P} in $[0,t_n+T_1]$,
contradicting the maximality of $u$ for $n$ large enough.
\end{proof}
We now recall from \cite{stable} the additional assumption on $f$
needed to set--up the potential well theory. \numeraf
\item \label{F2} There is $c_3>0$ such that
$$F(x,u)\le \frac {c_3}p|u|^p$$
for all $x\in\Omega$ and $u\in\R$, where $F$ is the primitive of $f$
defined in \eqref{F}. \finenumeraf 1

\begin{rem}\label{remnew3}
It is clear, recalling Remark~\ref{remnew2}, that $f_0$ given in
\eqref{f0} satisfies \fassref{F1} and \fassref{F2} when $2\le q<p$,
$a\le 0$ and $b\in \R$.
\end{rem}

 We set, when $2<p\le 2^*$,
\begin{equation}
K_0=\sup_{u\in H^1_{\Gamma_0}(\Omega),\,\,u\not=0} \frac
{\int_\Omega F(\cdot,u)}{\|\nabla u\|_2^p}. \label{K0}\end{equation}
By \fassref{F1} and \fassref{F2}, we have $0\le K_0\le
p^{-1}c_3B_1^p$, where $B_1$ is the optimal constant of the Sobolev
embedding $H^1_{\Gamma_0}(\Omega)\hookrightarrow L^p(\Omega)$, i.e.
\begin{equation}\label{Poincare}
B_1=\sup\limits_{u\in H^1_{\Gamma_0}(\Omega),\,\,u\not=0} \dfrac
{\|u\|_p}{\|\nabla u\|_2}.
\end{equation}
 We denote
 \begin{footnote}{this is the correct form of
$\lambda_1$, which is the unique positive maximum point of the
function $\lambda^2/2-K_0\lambda^p$,
 incorrectly typewritten in \cite{stable}} \end{footnote}
\begin{gather}\label{defE1}
\lambda_1=(1/pK_0)^{1/(p-2)}, \qquad E_1=\left(\frac 12-\frac
1p\right)\lambda_1^2, \\
\intertext{when $K_0>0$, while $\lambda_1=E_1=+\infty$ when $K_0=0$,
and} \label{defW}W=\{(u_0,u_1)\in H^1_{\Gamma_0}(\Omega)\times
L^2(\Omega): E(u_0,u_1)<E_1\quad \text{and}\quad \|\nabla
u_0\|_2>\lambda_1\} \intertext{where, in accordance to \eqref{eff},}
E(u_0,u_1):=\frac 12 \|u_1\|_2^2+\frac 12\|\nabla u_0\|_2^2-
\int_\Omega F(\cdot,u_0).
\end{gather}
Clearly when $K_0=0$ then $W=\emptyset$, so what follows is of
interest only when $K_0>0$.   On the other hand when $K_0=0$ all
weak solutions are global (see \cite[p. 389]{stable}).  We recall
the following result (\cite[Lemma~2, (ii)]{stable}).

\begin{lem}\label{lemma1}
Suppose that the assumptions of Theorem~\ref{localexistencetheorem},
together with \fassref{F2}, hold true. Let $u$ be the maximal
solution of \eqref{P}. Assume moreover that $(u_0,u_1)\in W$. Then
there is $\lambda_2>\lambda_1$ such that $\|\nabla u(t)\|_2\ge
\lambda_2$ and $\|u(t)\|_p\ge (pK_0/c_3)^{1/p} \lambda_2$ for all
$t\in [0,T_{\text{max}})$.
\end{lem}

Our final assumptions are the following ones. \numera
\item There is $c_4>0$ such that
$$Q(x,v)v\ge c_4\alpha(x) \left(|v|^\mu+|v|^m\right),\qquad 1<\mu\le m,$$\label{Q3}
for all $x\in\Gamma_1$, $v\in \R$; \finenumera 1 \numeraf
\item there is $\eps_0>0$ such that for all $\eps\in (0,\eps_0]$
there exists $c_5=c_5(\eps)>0$ such that
$$f(x,u)u-(p-\eps)F(x,u)\ge c_5 |u|^p$$
for all $x\in\Omega$, $u\in\R$. \label{F3} \finenumeraf 1

\begin{rem}\label{finalremark} Clearly $Q_0$ given in \eqref{Q0}
satisfies, beside \assref{Q1}--\assref{Q2}\begin{footnote}{as noted
in Remark~\ref{remnew1}}\end{footnote}, also \assref{Q3} with
$c_4=1$. Moreover \assref{Q3} immediately follows from \eqref{low}
when $m=\mu$, while it is not a consequence of
\assref{Q1}--\assref{Q2} when $\mu<m$. Next  $f_0$ given in
\eqref{f0} satisfies, beside
\fassref{F1}--\fassref{F2}\begin{footnote}{see
Remark~\ref{remnew3}}\end{footnote}, also \fassref{F3} when $a\le 0$
and $b>0$, with $\eps_0=p-q>0$ and $c_5(\eps)=b\eps/p$. Next
\fassref{F3} implies the standard growth condition
\begin{equation}\label{quadr}
f(x,u)u\ge p F(x,u)\qquad\text{for all $x\in\Omega$, $u\in \R$.}
\end{equation}
Finally it is worth observing that \fassref{F1}--\fassref{F2} and
\eqref{quadr} cannot be responsible of a blow--up phenomenon, since
 $f\equiv 0$ satisfies them and blow--up does not occur in this case.
\end{rem}

\section{Main result}\label{section3}
\bigskip
\noindent This section is devoted to state and prove our main
result. We start with a key estimate.
\begin{lem}\label{lemma2} Let $1<m\le 1+p/2$ and $2<p\le 2^*$. Then
there is a positive constant $C_1=C_1(m,p,\Omega, \Gamma_0)$ such
that
\begin{equation}\label{mainestimate}
    \|u\|_{m,\Gamma_1}^m\le C_1 \|u\|_p^{m-1}\|\nabla
    u\|_2\qquad \text{for all $u\in H^1_{\Gamma_0}(\Omega)$}.
\end{equation}
\end{lem}

\begin{proof} We first consider the auxiliary non--homogeneous Neumann problem
\begin{equation}\label{phi}
\begin{cases}-\Delta w+w=0\qquad &\text{in $\Omega$}\\
\partial_\nu w=1\qquad &\text{on $\Gamma$.}
\end{cases}
\end{equation}
By Riesz--Fr\'echet theorem problem \eqref{phi} has a unique weak
solution, i.e. $w\in H^1(\Omega)$ such that
\begin{equation}\label{weakneumann}
 \int_\Omega \nabla w\nabla \phi+\int_\Omega w\phi=\int_\Gamma
 \phi\qquad\text{for all $\phi\in H^1(\Omega)$.}
\end{equation}
Moreover, since $\Omega$ is bounded and $C^{1,1}$, by
Agmon--Douglis--Nirenberg regularity estimate (here used in the form
stated in \cite[Theorem~2.4.2.7, p. 126]{grisvard}), we have $w\in
W^{2,q}(\Omega)$ for all $q>1$. It follows, by Morrey's Theorem
(\cite[Corollary~9.15, p. 285]{brezis2}), that $w\in
C^1(\overline{\Omega})$.

Now let $u\in H^1(\Omega)$. We claim that $|u|^m\in
W^{1,1}(\Omega)$. Since $m\le 2^*$, by Sobolev embedding theorem we
have $|u|^m\in L^1(\Omega)$. Moreover, by using the chain rule for
Sobolev function (see \cite[Theorem~2.2]{mm}),  we get that $|u|^m$
possesses a weak gradient $\nabla (|u|^m)=m|u|^{m-2}u\,\nabla u$.
Since $m\le 1+2^*/2$, using Sobolev embedding theorem again,  we
have $|u|^{m-2}u\in L^2(\Omega)$, hence by H\"{o}lder inequality we
get that $\nabla (|u|^m)\in [L^1(\Omega)]^n$ and
$${\|\nabla (|u|^m)\|}_1\le m\left(\int_\Omega
|u|^{2(m-1)}\right)^{1/2}{\|\nabla u\|}_2.$$
Since $2(m-1)\le p$ and
$\Omega$ is bounded it follows
\begin{equation}\label{Iroman}
{\|\nabla (|u|^m)\|}_1\le m|\Omega|^{\frac 12 -\frac
{m-1}p}\|u\|_p^{m-1}\|\nabla u\|_2,
\end{equation}
where $|\Omega|$ denotes the Lebesgue measure of $\Omega$. Our claim
is then proved. Consequently  (see \cite[Corollary~9.8~p.
277]{brezis2}) there is a sequence $(\phi_n)_n$ in
$C^\infty_c(\R^N)$ such that ${\phi_n}_{|\Omega}\to |u|^m$ in
$W^{1,1}(\Omega)$. By the trace theorem it follows that
${\phi_n}_{|\Gamma}\to {|u|^m}_{|\Gamma}$ in $L^1(\Gamma)$. Since in
particular $\phi_n\in H^1(\Omega)$ then \eqref{weakneumann} holds
with $\phi=\phi_n$ for $n\in\N$.
 Since $w,|\nabla
w|\in L^\infty(\Omega)$ we can pass to the limit as $n\to\infty$ and
get
\begin{equation}\label{IIroman}
\int_\Omega \nabla w \nabla(|u|^m)+\int_\Omega w |u|^m=\int_\Gamma
|u|^m.
\end{equation}
Combining \eqref{Iroman} and \eqref{IIroman} we have
$$\|u\|_{m,\Gamma}^m\le \|w\|_\infty  \|u\|_m^m+m\|\nabla w\|_\infty |\Omega|^{\frac 12 -\frac
{m-1}p}\|u\|_p^{m-1}\|\nabla u\|_2$$ for all $u\in H^1(\Omega)$.
Since $m\le p\le 2^*$ and $\Omega$ is bounded,  we consequently get
by using H\"{o}lder inequality again
$$\|u\|_{m,\Gamma}^m\le \left(\|w\|_\infty |\Omega|^{1 -\frac
mp} \|u\|_p+m\|\nabla w\|_\infty |\Omega|^{\frac 12 -\frac
{m-1}p}\|\nabla u\|_2\right)\|u\|_p^{m-1}.$$ By restricting now to
$u\in H^1_{\Gamma_0}(\Omega)$ we use the Poincar\`e type inequality
recalled above to get \eqref{mainestimate}, where $C_1$ is given by
$$C_1=\|w\|_\infty |\Omega|^{1 -\frac
mp} B_1+m\|\nabla w\|_\infty |\Omega|^{\frac 12 -\frac {m-1}p},$$
where $B_1$ is the positive constant defined in \eqref{Poincare}.
Since $w$ depends only on $\Omega$, the proof is complete.
\end{proof}

We can finally state our main result.

\begin{thm}\label{theorem 4}
Suppose that \assref{Q1}--\assref{Q3} and \fassref{F1}--\fassref{F3}
hold, that $\alpha\in L^\infty(\Gamma_1)$,
$$2<p\le 1+2^*/2,\qquad 1<m<1+p/2,$$
and $(u_0,u_1)\in W$.  Then for any solution of \eqref{P} we have
$T_{max} < \infty$ and $\|u(t)\|_p\to \infty$ (so also
$\|u(t)\|_\infty\to \infty$ and $\|\nabla u(t)\|_2\to \infty$) as
$t\to T_{max}^-$.
\end{thm}

\begin{proof} The proof is a variant of the proof of
\cite[Theorem~7]{stable}, where we use Lemma~\ref{lemma2} instead of
the estimate \cite[(50)]{stable}. Nevertheless, since the proof of
\cite[Theorem~7]{stable} was itself a variant of the proof of
\cite[Theorem~2]{blowup}, we give in the sequel, for the sake of
clearness, a self--contained proof.

We first claim that our statement reduces to prove that problem
\eqref{P} cannot have global weak solutions, i.e. weak solutions in
the whole of $[0,\infty)$. Indeed, once this fact is proved, then we
must have, by Theorem~\ref{continuation}, that
$T_{\text{max}}<\infty$ and
\begin{equation}\label{new1}
\|u(t)\|_{H^{1}_{\Gamma_{0}}(\Omega)}+\|u_t(t)\|_2 \to
\infty\qquad\text{as $t\to T_{max}^-$}.
\end{equation}
Hence, to prove our claim, we have to show only that also
$\|u(t)\|_p\to \infty$ as $t\to T_{max}^-$. We first note that, by
\eqref{low} and \eqref{EI}, the energy function $E$ (defined in
\eqref{eff}) is decreasing. Hence, by \eqref{eff},
\begin{equation}\label{new2}
    \frac 12 \|\nabla u(t)\|_2^2 +\frac 12
    \|u_t(t)\|_2^2-\int_\Omega F(x, u(t))\le E_0
\end{equation}
for $t\in [0,T_{\text{max}})$, where $E_0:=E(u_0,u_1)$. Hence, by
\fassref{F2}, we have
\begin{equation}\label{new3}
    \frac 12 \|\nabla u(t)\|_2^2 +\frac 12
    \|u_t(t)\|_2^2-\frac {c_3}p\|u(t)\|_p^p\le E_0
\end{equation}
for $t\in [0,T_{\text{max}})$. Consequently, by \eqref{new1}, we get
that $\|u(t)\|_p\to \infty$ as well, so concluding the proof of our
claim.

We now have to prove that problem \eqref{P} cannot have global
solutions. We suppose by contradiction that $T_{\text{max}}=\infty$.
We  fix $E_2\in (E_0,E_1)$ and we set
\begin{equation}\label{calH}
\cal{H}(t)=\cal{H}(u(t),u_t(t))=E_2-E(u(t),u_t(t)).
\end{equation}
Since, as noted before, $E$ is decreasing, the function $\cal{H}$ is
increasing and $\cal{H}(t)\ge \cal{H}_0:=\cal{H}(0)=E_2-E_0>0$. In
the sequel of the proof we shall omit, for simplicity, explicit
dependence on time of $u$ and $u_t$ on the notation. By
Lemma~\ref{lemma1} we have
$$\cal{H}(t)\le E_2-\frac 12 \|\nabla u\|_2^2+
\int_\Omega F(\cdot,u)\le E_1-\frac 12\lambda_1^2+ \int_\Omega
F(\cdot,u)$$ and then, by \eqref{defE1} and \fassref{F3},
\begin{equation}\label{Hlow}
\cal{H}(t)\le \int_\Omega F(\cdot,u)\le \dfrac {c_3}p\|u\|_p^p.
\end{equation}

We now introduce, as in  \cite{georgiev} and \cite{levserr}, the
main auxiliary function which shows the blow--up properties of $u$,
i.e.
\begin{equation}\label{Z}
    \cal{Z}(t)=\cal{H}^{1-\eta}(t)+\xi\int_\Omega
u_t u,
\end{equation}
 where $\xi>0$ and $\eta\in (0,1)$ are  constants to
be fixed later. In order to estimate the derivative of $\cal{Z}$ it
is convenient to estimate
\begin{equation}\label{I}I_1:=\frac d{dt}\int_\Omega u_tu.
\end{equation}
Using  Definition~\ref{def1} we can take $\varphi=u$ in \eqref{ovo}
and get
\begin{equation}\label{II}
    I_1=\|u_t\|_2^2-\|\nabla u\|_2^2+\int_\Omega
f(\cdot,u)u-\int_{\Gamma_1}Q(\cdot,u_t)u
\end{equation}
almost everywhere in $(0,\infty)$.  Now we claim that there are
positive constants $c_6$ and $c_7$, depending on $p$ and $K_0$, such
that
\begin{equation}\label{III}
I_1\ge 2\|u_t\|_2^2+c_6\|u\|_p^p+c_7\|\nabla
u\|_2^2+2\cal{H}(t)-\int_{\Gamma_1}Q(\cdot,u_t)u
\end{equation}
in $[0,\infty)$. Using \eqref{eff} and \eqref{calH} we can write,
for any $\eps>0$,  the identity \eqref{II} in the form
\begin{multline}\label{IIbis}
I_1=\tfrac12 (p+2-\eps))\|u_t\|_2^2+\tfrac 12(p-2-\eps)\|\nabla
u\|_2^2\\+\int_\Omega
[f(\cdot,u)u-(p-\eps)F(\cdot,u)]+(p-\eps)\cal{H}(t)-(p-\eps)E_2-\int_{\Gamma_1}Q(\cdot,u_t)u.
\end{multline}
Using \fassref{F3}  for $0<\eps<\min\{\eps_0,p-2\}$ we consequently
get
\begin{align*}
I_1\ge &2\|u_t\|_2^2+\int_\Omega [f(\cdot,u)u -(p-\eps)F(\cdot,u)] +\tfrac
12(p-\eps-2)\|\nabla u\|_2^2
 -(p-\eps)E_2\\
& \hskip 6.6truecm +(p-\eps)\cal{H}(t) -\int_{\Gamma_1}Q(\cdot,u_t)u\\
\ge &2\|u_t\|_2^2+c_5(\eps)\|u\|_p^p+\tfrac 12(p-\eps-2)\|\nabla
u\|_2^2 -(p-\eps)E_2+2\cal{H}(t)-\int_{\Gamma_1}Q(\cdot,u_t)u.
\end{align*}
By Lemma~\ref{lemma1}
\begin{gather*}
\tfrac 12 (p-\eps-2)\|\nabla u\|_2^2-(p-\eps)E_2
\ge c_7(\eps)\|\nabla u\|_2^2+c_8(\eps),\\
\intertext{where} c_7(\eps)=\tfrac 12
(p-\eps-2)\left(1-\lambda_1^2/\lambda_2^2\right)\qquad\text{and}\quad
c_8(\eps)=\tfrac 12(p-\eps-2)\lambda_1^2-(p-\eps)E_2.
\end{gather*}
Clearly $c_7(\eps)>0$ and, as $\eps\to 0^+$,
$$c_8(\eps)\to\frac 12(p-2)\lambda_1^2-pE_2>\tfrac 12(p-2)\lambda_1^2-pE_1=0,$$
so also $c_8(\eps)>0$ for $\eps$ sufficiently small. Fixing a
sufficiently small $\eps=\overline{\eps}$ and setting
$c_6=c_5(\overline{\eps})$, $c_7=c_7(\overline{\eps})$ we conclude
the proof of \eqref{III}.

Now, in order to estimate $I_1$, we estimate the last term in
\eqref{III}.  Using \assref{Q2}, H\H{o}lder inequality (with respect
to $\mu_\alpha$), and assumption $\alpha\in L^\infty(\Gamma_1)$ we
obtain
$$I_2:=\left|\int_{\Gamma_1}Q(\cdot,u_t)u\right|\le
c_1
\|\alpha\|_{\infty,\Gamma_1}\left(\|u_t\|_{\mu,\Gamma_1,\alpha}^{\mu-1}\|u\|_{\mu,\Gamma_1}+
\|u_t\|_{m,\Gamma_1,\alpha}^{m-1}\|u\|_{m,\Gamma_1} \right).$$ Since
$\mu\le m$,  applying H\H{o}lder inequality again we get
\begin{equation}\label{IIIbis}
I_2\le C_2 \left(\|u_t\|_{\mu,\Gamma_1,\alpha}^{\mu-1}
+\|u_t\|_{m,\Gamma_1,\alpha}^{m-1}\right)\|u\|_{m,\Gamma_1}
\end{equation}
with
$C_2=C_2\left(\mu,m,c_1,\|\alpha\|_{\infty,\Gamma_1},\sigma(\Gamma_1)\right)>0$.
By  Lemma~\ref{lemma2} we consequently get
\begin{equation}\label{IV}
I_2\le C_3 \left(\|u_t\|_{\mu,\Gamma_1,\alpha}^{\mu-1}
+\|u_t\|_{m,\Gamma_1,\alpha}^{m-1}\right)\|u\|_p^{1-1/m}\|\nabla
u\|_2^{1/m}
\end{equation}
where
$C_3=C_3(\mu,m,p,c_1,\|\alpha\|_{\infty,\Gamma_1},\Omega,\Gamma_0)>0$.
Let us denote $$I_3:=\|u_t\|_{\mu,\Gamma_1,\alpha}^{\mu-1}
\|u\|_p^{1-1/m}\|\nabla u\|_2^{1/m}\qquad\text{and}\quad
I_4:=\|u_t\|_{\mu,\Gamma_1,\alpha}^{m-1} \|u\|_p^{1-1/m}\|\nabla
u\|_2^{1/m}.$$ It is convenient to write
\begin{equation}\label{V}
    I_3=\|u_t\|_{\mu,\Gamma_1,\alpha}^{\mu-1}\|\nabla u\|_2^{1/m}
\|u\|_p^{p\left(\frac 1\mu-\frac 1{2m}\right)} \|u\|_p^{1-\frac
1m-p\left(\frac 1\mu-\frac 1{2m}\right)}.
\end{equation}
We now apply, for any $\delta>0$, weighted Young's inequality to the
first three multiplicands in the right hand side of \eqref{V} , with
exponents $p_1=\mu'$, $p_2=2m$ and $p_3=2m\mu/(2m-\mu)$, so that
$\frac 1{p_1}+\frac 1{p_2}+\frac 1{p_3}=1$ (note that trivially
$p_1,p_2>1$ while $p_3>1$ as $\frac 1{p_3}=\frac 1\mu-\frac 1{2m}\in
(0,1)$ since $m\ge\mu>1$). Thus we get the estimate
\begin{equation}\label{VI}
I_3\le \left(\delta ^{\frac
1{1-\mu}}\|u_t\|_{\mu,\Gamma_1,\alpha}^\mu+\delta \|\nabla
u\|_2^2+\delta \|u\|_p^p\right)\|u\|_p^{1-\frac 1m-p\left(\frac
1\mu-\frac 1{2m}\right)}
\end{equation}
and, by particularizing it to the subcase $m=\mu$, also the estimate
\begin{equation}\label{VIbis}
I_4\le \left(\delta ^{\frac
1{1-m}}\|u_t\|_{m,\Gamma_1,\alpha}^m+\delta \|\nabla u\|_2^2+\delta
\|u\|_p^p\right)\|u\|_p^{1-\frac 1m-\frac p{2m}}.
\end{equation}
Moreover, by Lemma~\ref{lemma1} we have $\|u\|_p\ge
[c_3(pK_0)^{\frac 2{p-2}}]^{-1/p}$. Hence, since $\mu\le m$ implies
$1-\frac 1m-p\left(\frac 1\mu-\frac 1{2m}\right)\le 1-\frac 1m-\frac
p{2m}$,  we also have
\begin{equation}\label{upp}
\|u\|_p^{1-\frac 1m-p\left(\frac 1\mu-\frac 1{2m}\right)}\le
[c_3(pK_0)^{\frac 2{p-2}}]^{\frac 1\mu-\frac 1m} \|u\|_p^{1-\frac
1m-\frac p{2m}}.
\end{equation}
By combining \eqref{IV} and  \eqref{VI}--\eqref{upp} we get
\begin{equation}\label{VII}
I_2\le C_4\left[S(\delta)\left(
\|u_t\|_{\mu,\Gamma_1,\alpha}^\mu+\|u_t\|_{m,\Gamma_1,\alpha}^m
\right) +\delta \|\nabla u\|_2^2 +\delta
\|u\|_p^p\right]\|u\|_p^{1-\frac 1m-\frac p{2m}}
\end{equation}
where $S(\delta)=\left(\delta ^{\frac 1{1-\mu}}+\delta ^{\frac
1{1-m}}\right)$ and
$C_4=C_4(\mu,m,p,c_1,c_3,K_0,\|\alpha\|_{\infty,\Gamma_1},\Omega,\Gamma_0)>0$.
Now we set $\overline{\eta}=-\frac 1p \left(1-\frac 1m-\frac p{2m}
\right)$. Since $m<1+p/2$, we have $\overline{\eta}>0$. Moreover
$\overline{\eta}=\frac 1{2m}-\frac{m-1}{pm}<\frac 1{2m}<1$. By
combining \eqref{VII} and \eqref{Hlow} we get
\begin{equation}\label{VIII}
I_2\le C_5\left[S(\delta)\left(
\|u_t\|_{\mu,\Gamma_1,\alpha}^\mu+\|u_t\|_{m,\Gamma_1,\alpha}^m
\right) +\delta \|\nabla u\|_2^2 +\delta
\|u\|_p^p\right]{\cal{H}}^{-\overline{\eta}}(t)
\end{equation}
where
$C_5=C_5(\mu,m,p,c_1,c_3,K_0,\|\alpha\|_{\infty,\Gamma_1},\Omega,\Gamma_0)>0$.
Since,  by \eqref{EI} and  \assref{Q3} we have
$$\cal{H}'(t)\ge c_4\left(
\|u_t\|_{\mu,\Gamma_1,\alpha}^\mu+\|u_t\|_{m,\Gamma_1,\alpha}^m\right)$$
and $\cal{H}(t)\ge {\cal H}_0$, by \eqref{VIII} we get, for any
$\eta\in (0,\overline{\eta})$,
\begin{equation}\label{IX}
I_2\le C_6\left[S(\delta){\cal H}'(t){\cal{H}}(t)^{-\eta} +\delta
\|\nabla u\|_2^2 +\delta \|u\|_p^p\right]
\end{equation}
where
$C_6=C_6(\mu,m,p,c_1,c_3,K_0,\|\alpha\|_{\infty,\Gamma_1},\Omega,\Gamma_0,
{\cal H}_0)>0$. By combining \eqref{III} and \eqref{IX} we have the
desired estimate of $I_1$, i.e.
\begin{equation}\label{X}
    I_1\ge 2\|u_t\|_2^2+(c_6-\delta C_6)\|u\|_p^p+(c_7-\delta C_6)\|\nabla
    u\|_2^2+2{\cal{H}}(t)-S(\delta){\cal H}'(t){\cal{H}}^{-\eta}(t).
    \end{equation}
By making the choice $\delta=\text{min}\{c_6,c_7\}/(2C_6)$ from
\eqref{X} we get
\begin{equation}\label{XI}
    I_1\ge 2\|u_t\|_2^2+\frac{c_6}2\|u\|_p^p+\frac{c_7}2\|\nabla
    u\|_2^2+2{\cal{H}}(t)-C_7{\cal H}'(t){\cal{H}}^{-\eta}(t)
    \end{equation}
where
$C_7=C_7(\mu,m,p,c_1,c_3,K_0,\|\alpha\|_{\infty,\Gamma_1},\Omega,\Gamma_0,
{\cal H}_0)>0$.

By combining \eqref{Z} and \eqref{XI} we get, for any $\eta\in
(0,\overline{\eta})$,
$${\cal Z}'(t)\ge (1-\eta-C_7\xi){\cal H}^{-\eta}(t){\cal H}'(t)+2\xi{\cal{H}}(t)+2\xi\|u_t\|_2^2+\frac{\xi c_6}2\|u\|_p^p+\frac{\xi c_7}2\|\nabla
    u\|_2^2.$$
We now fix
$\eta=\text{min}\left\{\frac{\overline{\eta}}4,\frac{p-2}{4p}\right\}\in(0,1)$
and we restrict to $0<\xi\le(1-\eta)/{C_7}$. Hence, since
$\cal{H}'\ge 0$,  from previous estimate it follows
\begin{equation}\label{XII}
{\cal Z}'(t)\ge \xi c_8\left(\|u_t\|_2^2+\|\nabla
u\|_2^2+\|u\|_p^p+{\cal H}(t)\right)
\end{equation}
were $c_8=c_8(p,K_0)>0$. Next, since $\cal Z(0)={\cal
H}_0^{1-\eta}+\xi \int_\Omega u_0 u_1$, by fixing
$\xi=\xi_0=\xi_0(\mu,m,p,c_1,c_3,K_0,\|\alpha\|_{\infty,\Gamma_1},\Omega,\Gamma_0,u_0,u_1)>0$
sufficiently small we have $\cal Z(0)>0$, hence ${\cal Z}(t)\ge
{\cal Z(0)}>0$ by \eqref{XII}. Now we denote $r=1/(1-\eta)$ and
$\overline{r}=1/(1-\overline{\eta})$. Since
$0<\eta<\overline{\eta}<1$ we have $1<r<\overline{r}$. Now using
Cauchy--Schwartz inequality as well as the elementary inequality
$(A+B)^r\le 2^{r-1}(A^r+B^r)$ for $A,B\ge 0$, we have from \eqref{Z}
$${\cal Z}^r(t)\le \left({\cal H}^{1-\eta}(t)+\xi_0 \left|\int_\Omega u_tu\right| \,\right)^r
\le 2^{r-1}\left({\cal H}(t)+\xi_0^r\|u_t\|_2^r \|u\|_2^r\right).$$
We now set $q=2/r=2(1-\eta)$. Since $\eta<\frac 12-\frac 1p<\frac
12$ it follows that $q>1$. We can then apply Young's inequality with
exponents $q$ and $q'=\frac{1-\eta}{\frac 12-\eta}$ to get
$${\cal Z}^r(t)\le 2^{r-1}\left({\cal H}(t)+\xi_0^2\|u_t\|_2^2+ \|u\|_2^{\frac 1{\frac 12-\eta}}\right).$$
Now, since ${\frac 1{\frac 12-\eta}}<p$ a further application of
Young's inequality yields
$$\|u\|_2^{\frac 1{\frac 12-\eta}}\le
1+\|u\|_2^p$$ and then, as $\Omega$ is bounded and $\cal{H}(t)\ge
\cal{H}_0$, by H\"{o}lder inequality we get
\begin{equation}\label{XIII}
    {\cal Z}^r(t)\le C_8 \left({\cal H}(t)+\|u_t\|_2^2+ \|u\|_p^p\right)
\end{equation}
where
$C_8=C_8(\mu,m,p,c_1,c_3,K_0,\|\alpha\|_{\infty,\Gamma_1},\Omega,\Gamma_0,u_0,u_1)>0$.
By combining \eqref{XII} and \eqref{XIII}, as $r>1$, we get
$${\cal Z}'(t)\ge C_9 {\cal Z}^r(t)\qquad\text{for all $t\in
[0,\infty)$}$$ where
$C_9=C_9(\mu,m,p,c_1,c_3,K_0,\|\alpha\|_{\infty,\Gamma_1},\Omega,\Gamma_0,u_0,u_1)>0$.
Since $r>1$ this final estimate gives the desired contradiction.
\end{proof}

\section{Proof of Theorem~\ref{theorem1}}\label{section4}
\bigskip

This section is devoted to show that Theorem~\ref{theorem1} is a
simple corollary of Theorem~\ref{theorem 4}.  We first need to show
that, for problem~\eqref{2}, $E_1$ and $W$, as defined in
\eqref{defE1}--\eqref{defW}, are nothing but $d$ and $W_u$
(introduced in \eqref{d}--\eqref{Wu}). The proof is an adaptation of the proof of \cite[Lemma~4.1]{fisvit}.

\begin{lem} \label{NUOVO}
Suppose $f(x,u)=|u|^{p-2}u$, $2<p\leq 2^*$, $\sigma(\Gamma_{0})>0$.
Then $E_{1}=d$ and $W=W_u$.
\end{lem}

\begin{proof}
When $f(x,u)=|u|^{p-2}u$  we have $K_0=\frac 1p B_1^p$, hence
\begin{equation}\label{gho1}
\lambda_1=B_1^{\frac {-p}{p-2}}\quad\text{and}\quad E_1=\left(\frac
12-\frac 1p\right) B_1^{-2p/(p-2)}.
\end{equation}
An easy calculation shows that for any $u\in
H^{1}_{\Gamma_{0}}(\Omega)\setminus\{0\}$ we have
    $$\max\limits_{\lambda > 0}J(\lambda u)=J(\lambda(u)u)=\left(\frac{1}{2}-\frac{1}{p}\right)\left(\frac{\left\|\nabla u\right\|_{2}}{\left\|u\right\|_{p}}
    \right)^{2p/(p-2)}\negqquad\text{where}\quad
    \lambda(u)=\frac{\left\|\nabla
    u\right\|^{2/(p-2)}_{2}}{\left\|u\right\|^{p/(p-2)}_{p}}.$$
Hence, by \eqref{Poincare}, $d=\left(\frac 12-\frac 1p\right)
B_1^{-2p/(p-2)}$. Combing with \eqref{gho1} we have $d=E_1$.

In order to show that $W=W_u$ we first prove that $W\subseteq W_u$.
Let $(u_0,u_1)\in W$ and suppose, by contradiction, that $K(u_{0})>
0$. Hence $\left\|u_{0}\right\|^{p}_{p}<\left\|\nabla
u_{0}\right\|^{2}_{2}$ by \eqref{J}. Moreover, $J(u_{0})\le
E(u_0,u_1) < d =E_{1}$ and $\left\|\nabla u_{0}\right\|_{2}>
\lambda_{1}$. Then it follows that
    \[E_{1}>E(u_0,u_1)\ge J(u_0)>\left(\frac{1}{2}-\frac{1}{p}\right)\left\|\nabla u_{0}\right\|^{2}_{2} >\left(\frac{1}{2}-\frac{1}{p}\right)\lambda^{2}_{1},
\]
which contradicts \eqref{defE1}.

To prove that $W_u\subseteq W$, we take $(u_0,u_1)\in W_{u}$. We
note that, by \eqref{Poincare}, we have $J(v)\geq h(\|\nabla v\|_2)$
for all $v\in H^1_{\Gamma_0}(\Omega)$, where $h$ is defined by
$h(\lambda)=\frac 12 \lambda^2-\frac 1p B_1^p\lambda^p$ for
$\lambda\geq 0$. One easily verify that $h(\lambda_1)=E_1$. Then,
since $J(u_0)\le E(u_0,u_1)<E_1$, we have $\|\nabla
u_0\|_2\not=\lambda_1$. Moreover, since $K(u_0)\le 0$, by
\eqref{Poincare} we have
$$\|\nabla u_0\|_2^2\le \|u_0\|_p^p\le
B_1^p\|\nabla u_0\|_p^p$$ and consequently $\|\nabla u_0\|_2\ge
B_1^{-p/(p-2)}=\lambda_1$. Then $\|\nabla u_0\|_2>
B_1^{-p/(p-2)}=\lambda_1$,  concluding the proof.
\end{proof}

\begin{rem}\label{variational}  When $f(x,u)=|u|^{p-2}u$ $d$ is also equal to the Mountain Pass level
associated to the elliptic problem
$$\begin{cases} -\Delta u=|u|^{p-2}u \qquad &\text{in
$\Omega$,}\\
 u=0\qquad &\text{on $\Gamma_0$,}\\
 \partial_\nu u=0 \qquad &\text{on $\Gamma_1$,}
\end{cases}$$
that is $d=\inf\limits_{\gamma\in \Lambda}\sup\limits_{t\in
[0,1]}J(\gamma(t))$, where
$$\Lambda=\{\gamma \in C([0,1];H^1_{\Gamma_0}(\Omega)): \gamma(0)=0,\quad
J(\gamma(1))<0\}.$$ The proof of this remark was given in
\cite[Final Remarks]{poroso}.
\end{rem}

We can now give the
\begin{proof}
[\bf Proof of Theorem~\ref{theorem1}] By Remark~\ref{finalremark}
the nonlinearities involved in problem~\eqref{2} satisfy assumption
\assref{Q1}--\assref{Q3} and \fassref{F1}--\fassref{F3}, so we can
apply Theorem~\ref{theorem 4}. Due to Lemma~\ref{NUOVO} we get
exactly Theorem~\ref{theorem1}.
\end{proof}
\def\cprime{$'$}
\providecommand{\bysame}{\leavevmode\hbox
to3em{\hrulefill}\thinspace}
\providecommand{\MR}{\relax\ifhmode\unskip\space\fi MR }
\providecommand{\MRhref}[2]{%
  \href{http://www.ams.org/mathscinet-getitem?mr=#1}{#2}
} \providecommand{\href}[2]{#2}

\end{document}